\documentclass[12pt]{amsart}
\usepackage[ansinew]{inputenc}
\usepackage{fancyhdr}
\usepackage{epsfig}
\usepackage{epic}
\usepackage{eepic}
\usepackage{amsfonts}
\usepackage{threeparttable}
\usepackage{amscd}
\usepackage{here}
\usepackage{graphicx}
\usepackage{lscape}
\usepackage{hyperref}
\usepackage{tabularx}
\usepackage{subfigure}
\usepackage{longtable}
\usepackage[pctex32]{color}
\usepackage{colortbl}
\usepackage{amsmath,amsthm}
\usepackage{amssymb,latexsym}
\usepackage{enumerate}
\usepackage{mathrsfs}
\usepackage{layout}
\usepackage[none]{hyphenat}

\usepackage{rotating} %Para rotar texto, objetos y tablas seite. No se ve en DVI solo en PS. Seite 328 Hundebuch
\newtheorem{lema}{Lemma}
\newtheorem{teor}{\sc {\textbf {Theorem}}}
\newtheorem{defin}{\sc {\textbf {Definition}}}
\newtheorem{coro}{\sc {\textbf {Corollary}}}

\newcommand{\Dr}{\mathbb{R}}
\newcommand{\F}{\mathcal{F}}
\newcommand{\Df}{\mathsf{F}}

\newcommand{\Dn}{\mathbb{N}}

\title[Sectional Connecting Lemma]{Sectional Connecting Lemma}
\author[S. Bautista, V. Sales, Y. Sánchez. ]{ S. Bautista, V. Sales, Y. Sánchez.}

\address{S. Bautista \\
Departamento de Matem\'{a}ticas \\
Universidad Nacional de Colombia, Bogot\'{a}, Colombia}
\email{sbautistad@unal.edu.co}

\address{V. Sales \\
Departamento de Matem\'{a}tica \\
Universidade Federal do Maranhão, São Luis, Brasil}
\email{vsales1@gmail.com}

\address{Y. S\'{a}nchez \\
Departamento de Matem\'{a}ticas \\
Universidad Nacional de Colombia, Bogot\'{a}, Colombia}
\email{yasanchezr@unal.edu.co}

\date{\today}


\begin{document}

\sloppy 

\begin{abstract}
A hyperbolic set on a compact manifold $M$, satisfies the property: given two of your any points $p$ and $q$, such that for all positive $\epsilon>0$, there is a trajectory in the hyperbolic set from a point $\epsilon$-close to $p$ to a point
$\epsilon$-close to $q$, then there is a point in $M$ whose $\alpha$-limit is that of $p$ and whose $\omega$-limit is that of $q$. Bautista and Morales in \cite{bm1}, give a version of this property, for sectional-Anosov flows (vector fields whose maximal invariant set is sectional-hyperbolic), including some conditions; among them that limit the dimension of $M$ to  three. In this paper, we prove a generalization of this result, for sectional-hyperbolic sets of codimension one in high dimensions.\\
\end{abstract}

\maketitle

\section{Introduction}

The sectional-hyperbolic sets are a more general class than hyperbolic sets, since it includes these and other non-hyperbolic sets as geometric Lorenz attractor, then it's relevant study which  properties valid for hyperbolic sets are also satisfied by the sectional-hyperbolic sets.  A property of the Anosov flows (when the whole  manifold is a hyperbolic set), is the Anosov connecting lemma (Theorem \ref{teor1}), this result was extended by Bautista and Morales in \cite{bm1}, for sectional-Anosov flows (when the maximal invariant is a sectional-hyperbolic set), in dimension three, including some necessary conditions;  this result is known as Sectional- Anosov Connecting Lemma (Theorem \ref{teor3}).\\

Although the Anosov Connecting Lemma, it is very useful in hyperbolic dynamics, presents the limitation of requiring that the flow should be Anosov, however, thanks to the theory of invariant manifolds (see \cite{hps}), the same property can be obtained for arbitrary hyperbolic sets (Theorem \ref{teor2}). In this paper, our main objective is extend the sectional-Anosov connecting Lemma to  high dimensions without the limitation of that the flow should be sectional-Anosov, allowing to use it directly in sectional-hyperbolic sets that contain the unstable manifolds of their hyperbolic subsets. For this, we including some conditions, and also generalize the characterization of  omega-limit sectional-hyperbolic sets which are closed orbits, given by Bautista and Morales  \cite{bm2}, from dimension three to high dimensions. Below we will give the necessary definitions to specify our objective.\\

Hereafter $M$ will be a compact manifold possibly with nonempty boundary endowed  with a Riemannian metric $\langle\cdot,\cdot\rangle$  induced by norm $||\cdot||$. Given $X$ an $C^1$ vector field, inwardly transverse to the boundary (if nonempty) we call $X_t$ its induced {\em flow} on $M$. Define the {\em maximal invariant set} of $X$ by
$$
M(X)=\displaystyle\bigcap_{t\geq0}X_t(M).
$$
The \textit{orbit} of a point $p \in M(X)$ is defined by $\mathcal{O}(p)=\{X_t(p)\,|\,t\in\mathbb{R}\}$.
A {\em singularity}  is a point $q$  where $X$ is zero, i.e. $X(q)=0$ (or equivalently $\mathcal{O}(q)=\{q\}$)and a {\em periodic orbit} is an orbit $\mathcal{O}(p)$ such that $X_T(p)=p$ for some minimal $T>0$ and $\mathcal{O}(p)\neq \{p\}$. By a {\em closed orbit}
is a singularity or a periodic orbit.\\

Given $p\in M$ we define the {\em omega-limit set},
$\omega_X(p)=\{x\in M\,|\,x=\lim_{n\rightarrow\infty}X_{t_n}(p)$ for some sequence $t_n\rightarrow\infty\}$, if $p \in M(X)$, define the {\em alpha-limit set}
$\alpha_X(p)=\{x\in M:x=\lim_{n\to{\infty}} X_{-t_n}(p),
$ for some sequence $ t_n\to \infty\}$.\\

A compact subset $\Lambda$ of $M$ is called {\em invariant} if
$X_t(\Lambda)=\Lambda$ for all $t\in\mathbb{R}$; {\em transitive} if $\Lambda = \omega_X(p)$ for some $p\in \Lambda$. A compact invariant set $\Lambda$ is {\em attracting} if there is a neighborhood $U$ such that
\[\Lambda=\cap_{t \geq 0}X_t(U),\]
and an {\em attractor} of $X$,  is an  attracting set $\Lambda$  which is {\em transitive}. On the other hand, a compact invariant set $\Lambda$ is {\em Lyapunov stable}, if for every neighborhood $U$ of $\Lambda$, exists a neighborhood $W$ such that: $X_t(p)\in U$ for all $p\in W$ and $t\geq 0$.

\begin{defin}\label{defin1}
A compact invariant set $\Lambda \subseteq M(X)$ is {\em hyperbolic} if there are positive constants $K,\lambda$ and a continuous $DX_t$-invariant splitting of tangent bundle  $T_{\Lambda}M=E^s_{\Lambda}\oplus E^X_{\Lambda}\oplus E^u_{\Lambda}$,
 such that for every $x \in \Lambda$ and $t \geq 0$:
\begin{enumerate}
\item [$(1)$] $\| DX_t(x)v^s_x \| \leq K e^{-\lambda t}\| v^s_x \|,\ \ \forall v^s_x \in E^s_x$;
\item [$(2)$] $ \| DX_t(x)v^u_x \| \geq K^{-1} e^{\lambda t} \| v^u_x \|,\ \ \forall v^u_x \in E^u_x$;
\item [$(3)$] $E^{X}_{x}=\left\langle X(x)\right\rangle $.
\end{enumerate}
\end{defin}

If $E^s_x\neq 0$ and $E^u_x\neq 0$ for all $x\in \Lambda$ we will say that $\Lambda$
is a {\em saddle-type hyperbolic set}. A closed orbit is hyperbolic if it does as a compact invariant set of $X$.

When $\Lambda =M$, we say that the flow generated by $X$ is an Anosov flow.\\

The invariant manifold theory \cite{hps} asserts that if $H\subseteq M$ is hyperbolic set of $X$ and $p \in H$, then the topologic sets:
$$W^{ss}(p)=\{q \in M: \lim_{t\to\infty} d(X_t(q),X_t(p))=0\}$$
and
$$W^{uu}(p)=\{q \in M: \lim_{t\to- \infty} d(X_t(q),X_t(p))=0\}$$
they are $C^1$ manifolds in $M$, so-called strong stable and unstable manifolds, tangent at $p$ to the subbundles
$E^s_p$ and $E^u_p$ respectively. Saturating them with the flow we obtain the stable and unstable
manifolds $W^{s}(p)$ and $W^{u}(p)$ respectively, which are invariant. If $p,p' \in H$, we have to
$W^{ss}(p)$ and $W^{ss}(p')$ are same or disjoint (similarly for $W^{uu}$).\\

\begin{defin}\label{defin2}
A compact invariant set $\Lambda \subseteq M(X)$ is {\em sectional-hyperbolic} if every singularity in $\Lambda$ is hyperbolic (as invariant set) and there are a continuous $DX_t$-invariant splitting of tangent bundle $T_{\Lambda} M = \Df^s_{\Lambda}\oplus \Df^c_{\Lambda}$, and positive constants $K,\lambda$
such that for every $x \in \Lambda$ and $t \geq 0$:

\begin{enumerate}
\item [$(1)$]
$\| DX_t(x)v^s_x \| \leq K e^{-\lambda t}\| v^s_x \| ,\ \ \forall v^s_x \in \mathsf{F}^s_x$;
\item [$(2)$]
$\| DX_t(x)v^s_x \| \cdot \| v^c_x \|  \leq K e^{-\lambda t}
\| DX_t(x)v^c_x \| \cdot \| v^s_x \|,\ \ \forall v^s_x \in \mathsf{F}^s_x,\ \ \forall v^c_x \in \mathsf{F}^c_x$;
\item [$(3)$]
$ \| DX_t(x)u^c_x , DX_t(x)v^c_x  \|_{X_t(x)} \geq K^{-1} e^{\lambda t}  \| u^c_x , v^c_x \|_x, \ \
\forall u^c_x , v^c_x \in \mathsf{F}^c_x$.
Where $||\cdot, \cdot ||_x$  is induced $2$-norm by the Riemannian metrics
$\langle \cdot, \cdot \rangle_x$ of $T_x\Lambda$, given by
$$||v_x,u_x||_x=\sqrt{\langle v_x,v_x \rangle_x \cdot \langle u_x,u_x\rangle_x - \langle v_x,u_x
\rangle_x^2}$$
for all $x \in \Lambda$ and every $u_x,v_x\in T_x\Lambda$
\end{enumerate}
\end{defin}

The third condition guarantees, the increase exponential of the area of parallelograms in the central subbundle $\mathsf{F}^c$. Since $ X(x) \in \Df^c_x$ for all $x \in \Lambda$ (see lemma 4 in \cite{bm}), have that  dimension of the central subbundle must be greater than or equal to $2$. In the particular case where  $dim(\mathsf{F}^c_x)=2$ we will say that $\Lambda$ is a sectional-hyperbolic set of \textit{codimension $1$}. \\

When $\Lambda=M(X)$, we say that the flow generated by $X$ is an sectional-Anosov flow.\\

Also the invariant manifold theory \cite{hps} asserts that through any point $x$ of a sectional-hyperbolic set $\Lambda$ has the
strong stable manifolds $\F^{ss}(x)$, tangent at $x$ to the subbundle $\Df^s_x$, which induces an foliation over $\Lambda$; saturating them with the flow we obtain the invariant manifold $\F^s(x)$.\\

Unlike hyperbolic sets, the sectional-hyperbolic sets can have regular orbits accumulating singularities. We have to:

\begin{lema}\label{lema1}
If $\Lambda \subseteq M(X)$ is sectional-hyperbolic set, and $\sigma$ is an singularity in $\Lambda$ then:
$$\F^{ss}(\sigma) \cap \Lambda = \{\sigma\}$$
\end{lema}

\begin{proof}
See corollary 2 in \cite{bm}.
\end{proof}

All singularity $\sigma$ in an sectional-hyperbolic set, is hyperbolic, so your invariant manifolds $W^{uu}(\sigma)$ and $W^{ss}(\sigma)$ are well defined. The strong  stable manifold sectional $\F^{ss}(\sigma)$ it is a submanifold of $W^{ss}(\sigma)$, with respect to your dimension, exists two possibilities:

\begin{enumerate}
\item $dim(W^{ss}(\sigma)) = dim(\F^{ss} (\sigma))$, in this case $W^{ss}(\sigma) = \F^{ss} (\sigma)$;
\item $dim(W^{ss}(\sigma)) = dim(\F^{ss}(\sigma)) + 1$, in this case, we say that the singularity is {\em Lorenz-like}.
\end{enumerate}

Every singularity Lorenz-like is type-saddle hyperbolic set with at least two negative eigenvalues, one of which is real eigenvalue $\lambda$ with multiplicity one such that the real part of the other eigenvalues are outside the closed interval
$[\lambda, -\lambda]$.\\

Over a Lorenz-like singularity $\sigma \in \Lambda$, we have $F^{ss}(\sigma)$ is tangent to the subspace associated the eigenvalues with real part less than $\lambda$, and $\F^{ss}(\sigma)$ divide a $W^{ss}(\sigma)$ in two connected component. If $\Lambda$ intersect just one connected component of $W^{ss}(\sigma) \setminus \F^{ss}(\sigma)$, we say that the singularity Lorenz Like is {\em of boundary-type}.\\

We say that a cross section $\Sigma$ of $X$ is associated to a Lorenz-like singularity $\sigma$ in a sectional-hyperbolic set $\Lambda$, if $\Sigma$ is very close to $\sigma$, $\Sigma \cap \Lambda \neq \emptyset$ and one of the connected components of $W^{ss}(\sigma) \setminus \F^{ss}(\sigma)$ contains a point in $int(\Sigma)$.\\
Another important result about the sectional-hyperbolic sets, is the \emph{\textbf{hyperbolic lemma}} (see lemma 9 in \cite{bm}), which assert that any invariant subset $H$ without singularities of a sectional-hyperbolic set $\Lambda$,  is hyperbolic, in this case, we have to that $\Df^s_H=E^s_H$ and $\Df^c_H=E^u_H \oplus E^X_H$, so $W^{ss}(p)=\F^{ss}(p)$ for all $p \in H$.\\

Let $p,q\in M$, we say that $p\prec q$ if and only if for all $\epsilon>0$ exists a orbit from a point $\epsilon$-close to $p$ to a point $\epsilon$-close to $q$.

\begin{teor}[Anosov Connecting Lemma] \label{teor1}

If $X$ is an Anosov flow on a compact manifold $M$ and $p, q \in M $ satisfy that $p \prec q$, then there is a point $x \in M$, such that $\alpha(x)=\alpha(p)$ and  $\omega(x)=\omega(q)$.
\end{teor}

The following theorem is a generalization of the Anosov connecting lemma, which allows  to be used in hyperbolic sets, even when the flow is not Anosov.

\begin{teor}\label{teor2}
Let $H$ be a hyperbolic set of vector field $X$ on $M$. If $p, q \in  H$ and there are sequences $z_n \in H$, $t_n \in \Dr^+$ such that $z_n \to p$ and $X_{t_n}(z_n) \to q$, then there is $x \in M$ such that $\alpha(x)=\alpha(p)$ and $\omega(x)=\omega(q)$.
 \end{teor}

As previously mentioned, Bautista and Morales generalized the
Theorem \ref{teor1}, in the sectional-hyperbolic dynamics, for
sectional-Anosov flows in dimension three:

\begin{teor}[Sectional-Anosov Connecting Lemma]\label{teor3}
If $X$ is a sectional-Anosov flow on a compact 3-manifold $M$, $p \in M(X)$ and $q \in M$, satisfy that $p \prec q$, and $\alpha(p)$ don’t have singularities, then there is $x \in M$ such that $\alpha(x)=\alpha(p)$ and $\omega(x)=\omega(q)$ or $\omega(x)$ is a singularity.
\end{teor}

As the main result of this paper, we prove the following
generalization of the previous theorem:

\begin{teor}[Main: Sectional Connecting Lemma]\label{teor4}
Let $\Lambda$ a sectional-hyperbolic set of codimension $1$ of a vector field $X$ on $M$, such that $W^{u}(H) \subseteq \Lambda$ for all hyperbolic subset $H$ of $\Lambda$. If $p, q \in \Lambda$ satisfy that $p \prec q$ and $\alpha(p)$ don't have singularities,then there is $x \in M$ such that $\alpha(x)=\alpha(p)$ and $\omega(x)=\omega(q)$ or $\omega(x)$ is a singularity.
\end{teor}

Note that in the theorem \ref{teor4}, two of the hypotheses of the theorem \ref{teor3} are replaced by more general ones; specifically, it is not requested that $M$ be of dimension three, but that $\Lambda$ be a sectional-hyperbolic set of codimension one, and the hypothesis that $X$ is a Sectional-Anosov flow is changed by the condition that $\Lambda$ contains the unstable manifolds their hyperbolic subsets; these variations, generate a change in the proof, however, some of the arguments are similar.\\

As direct consequences of the main theorem we have that:

\begin{coro}\label{coro1}
Every sectional-hyperbolic Lyapunov stable set of codimension $1$ of a vector field $X$ over $M$, satisfy that if  $p,q \in \Lambda$, $p \prec q$ and $\alpha(p)$ don't have singularities, then there is $x \in M$ such that $\alpha(x)=\alpha(p)$ and $\omega(x)=\omega(q)$ or $\omega(x)$ is a singularity.
\end{coro}

\begin{coro}	\label{coro2}
Every sectional-Anosov flow of codimension $1$ of a vector field over $M$, satisfy that if  $p,q \in \Lambda$, $p \prec q$ and $\alpha(p)$ don't have singularities, then there is $x \in M$ such that $\alpha(x)=\alpha(p)$ and $\omega(x)=\omega(q)$ or $\omega(x)$ is a singularity.
\end{coro}

To proof our main theorem, in section 2, we will introduce the definition of sectional partition \footnote {This definition results from making a modification to the definition of the singular partition introduced in \cite{bm2}}, we will prove its existence for invariant compact sets; in section 3 we will proof some properties of these partitions on sectional-hyperbolic sets of codimension one; in section 4, we will use the sectional partitions to characterize of omega-limit sectional-hyperbolic sets which are closed orbits the transitive hyperbolic sectional sets in codimension one, which are closed orbits, and with this characterization, finally in section 5 we will proof the main theorem.

\section{Sectional Partition}

Denote by $\mathcal{R}'= \{S_1, S_2, ... S_k\}$ a finite collection of cross sections, then we define:
$$\mathcal{R}=\bigcup_{i=1}^kS_i \,\,\text{  ,  }\,\, \partial \mathcal{R}=\bigcup_{i=1}^k\partial
S_i \,\,\text{  ,  } \,\,int(\mathcal{R})=\bigcup_{i=1}^k int(S_i).$$
The diameter of $\mathcal{R}$ is given by the max of the diameters of the elements of
 $\mathcal{R}'$, we  say that is of the time  $\epsilon$ if
$\mathcal{R} \cap
X_{[-\epsilon,\epsilon]}(y)=\{y\}$ for all $y \in \mathcal{R}$.

\begin{defin}\label{defin3}
A {\em sectional partition} of a compact invariant set
 $H$ of $X$  is a  finite and disjoint collection of cross sections $\mathcal{R}'$ of $X$ with nonzero time, such that:
$$Sing(X)\cap H = \{y \in H: X_t(y) \notin int(\mathcal{R}),\forall t
\in \mathbb{R} \}.$$
\end{defin}

\begin{teor}[Existence of sectional partitions]\label{teor5}

Let $\Lambda$ be a compact and invariant set of the field $X$ over $M$. If $\Lambda$ is not a singularity and every  singularity of $\Lambda$ is hyperbolic, then for all $\delta>0$  there is a sectional partition of diameter less than $\delta$.
\end{teor}

\begin{proof}
Let $\delta>0$,  as all the singularities of  $\Lambda$ are hyperbolic, we know that $\Lambda$ has a finite number of singularities, so there exists  $\delta_0$ such that:
$$Sing(X)\cap \Lambda=\bigcap_{t\in\mathbb{R}}X_t \left( \bigcup_{\sigma\in
Sing(X)\cap \Lambda} B_{\delta_0} (\sigma)\right).$$
We define:
$$H=\Lambda \backslash \left( \bigcup_{\sigma\in Sing(X)\cap\Lambda}B_{\delta_0}(\sigma)\right).$$
If $H=\emptyset$, then $\Lambda$ does not contain regular  orbits, in this case it is a singularity, this contradicts the hypothesis, then $H\neq \emptyset$ and $H \cap Sing(X)=\emptyset$. Then for all $z\in H$  there is a cross section   $R_z$ with $z \in int(R_z)$,  of arbitrarily small diameter. We consider this  diameter much smaller than     $\delta$ and defined:
$$V_z=\bigcup_{t\in(-1,1)} X_t(int(R_z));$$
clearly $z \in V_z$, so we have that
$\{V_z: z\in H\}$  is an open covering of $H$, but since $H$ is
 compact (a closed set inside a compact set), there is
$\{z_1,...,z_r\}$ such that:
$$H\subseteq \bigcup_{i=1}^r V_{z_i}.$$
Let us consider the rectangles $R_{z_1},R_{z_2},...,R_{z_r}$, if necessary we can move them through the flow and assume that they are pairwise
 disjoint. Now observe that $$\mathcal{R}'= \{R_{z_1}, ... ,R_{z_r}\};$$
satisfies the conditions of sectional partition. If $x\notin \Lambda \cap Sing(X)$, then
$$x \notin \bigcap_{t\in\mathbb{R}}X_t \left( \bigcup_{\sigma\in
Sing(X)\cap \Lambda} B_{\delta_0} (\sigma)\right)=Sing(X)\cap\Lambda$$
so there is $t_0 \in \Dr$ such that
$$X_{t_0}(x) \notin  \bigcup_{\sigma\in
Sing(X)\cap \Lambda} B_{\delta_0} (\sigma)$$
therefore $X_{t_0}(x) \in H$  and    $X_{t_0}(x) \in
V_{z_i} $  for some $ 1 \leq i \leq r $. We can say that   $ X_{t_0}(x) = X_{t_1}(w)$  for some
 $-1<t_1<1$  and $ w \in int(R_{z_i}) $, so $ X_{t_0-t_1}(x) = w \in int(R_{z_i})\subseteq int(\mathcal{R})$, where
$$x \notin \{y \in \Lambda : X_t(y) \notin int(\mathcal{R}),\forall t
\in \mathbb{R} \}$$
Therefore
$$\{y \in \Lambda : X_t(y) \notin int(\mathcal{R}),\forall t
\in \mathbb{R} \} \subseteq Sing(X)\cap \Lambda$$
since the other containment is immediate,  we get the result.

\end{proof}

Given a sectional partition $\mathcal{R}'$ of a compact invariant set $H$ of a field $X$  we define the function
$$ \Pi_{(\mathcal{R},int(\mathcal{R}))}: Dom(\Pi_{(\mathcal{R},int(\mathcal{R})})\subseteq
\mathcal{R} \rightarrow int(\mathcal{R})$$ with
$$Dom(\Pi_{(\mathcal{R},int(\mathcal{R})})=\{x \in \mathcal{R}: X_t(x)\in int(\mathcal{R})
\text{ for some } t > 0\}$$
given by
$$\Pi_{(\mathcal{R},int(\mathcal{R}))}(x) = X_{t(x)}(x),$$
where $t(x)$ is the time of return, i.e., the first $t>0$ for which
$X_t(x)\in int(\mathcal{R})$. In the remainder of this section we shall represent  $\Pi_{(\mathcal{R},int(\mathcal{R}))}$ only by $\Pi$.\\

Given $x \in S_i \in \mathcal{R'}$ we define $B_{\epsilon}(x,\mathcal{R})=B_\epsilon  (x) \cap S_i$.

\begin{lema}\label{lema2}
Let $\mathcal{R}'$ a sectional partition of the invariant compact set
 $H$ of $X$,  with  all its hyperbolic singularities,   then we have the following properties:

\renewcommand{\labelenumi}{\arabic{enumi}.}
\begin{enumerate}
\item [$(1)$] $H\cap \mathcal{R} \cap Dom(\Pi)\subseteq
int(Dom(\Pi))$ in $\mathcal{R}$ and $\Pi$ is $C^1$ in a neighborhood of $H\cap int(\mathcal{R})$ in $\mathcal{R}$;
\item [$(2)$] $(H\cap \mathcal{R}) \backslash Dom(\Pi) \subseteq \bigcup_{\sigma \in
Sing(X)\cap H} W^s(\sigma).$
\end{enumerate}
\end{lema}

\begin{proof}
For simplicity, we  denote $H^0=H\cap \mathcal{R}$\\

\renewcommand{\labelenumi}{\arabic{enumi}.}
\begin{enumerate}
\item [$(1)$] Let $x \in H^0\cap Dom(\Pi)$, so $X_{t(x) }(x)\in int(S_i)$ for some $S_i\in \mathcal{R}'$ and $t(x)>0$, due to continuous dependence of the flow, there is $\epsilon_x>0$, such that $\mathcal{O}^+(y) \cap int(S_i)\neq \emptyset$ for all $y\in B_{\epsilon_x}(x)$, then $B_{\epsilon_x}(x,\mathcal{R}) \subseteq Dom(\Pi)$, thereby $x \in int(Dom (\Pi))$ in $\mathcal{R}$. Define
$$U = \bigcup_{x\in H^0\cap Dom(\Pi)}
B_{\epsilon_x}(x,\mathcal{R})$$ we have that $U$ is a neighborhood of $H^0$ in
$\mathcal{R}$ such that $\Pi$ is $C^1$.

\item [$(2)$] Let $p \in H^0 \backslash Dom(\Pi)$, thus $\mathcal{O}^+(p) \cap int(\mathcal{R}) = \emptyset$.  Suppose there is a regular point             $r \in \omega(p) \subseteq H$,  then by the definition of sectional partition, there exists  $t_0 \in \mathbb{R}$ such that $X_{t_0}(r)\in int(\mathcal{R})$, thus $X_{t_0}(r) \in int(S_j)$ for some
$S_j \in \mathcal{R}'$. Given that $X_{t_0}(r)\in \omega(p)$ there is a sequence $t_n
\rightarrow \infty$ such that $X_{t_n}(p) \rightarrow X_{t_0}(r)$ and since
 $S_j$ is transversal to the flow we have to   $\mathcal{O}^+(p)\cap int(S_j)\neq \emptyset$  this is a contradiction.
 Thus  $\omega(p)$ is a singularity and $p \in \bigcup_{\sigma \in Sing(X)\cap H} W^s(\sigma)$.
\end{enumerate}
\end{proof}

\begin{lema}\label{lema3}
Given $q \in M$ if $\omega(q)$ is not a singularity and
$\mathcal{R}'$  is a sectional partition of  $\omega(q)$,   so we have that
 $\mathcal{O}^+(q)\cap int(\mathcal{R})=\{q_1,q_2,...\}$  is an infinite sequence ordered in  a way that
 $\Pi(q_n)=q_{n+1}$.\\
\end{lema}

\begin{proof} Since the singularities are isolated and $\omega(q)$  is not a singularity, then
 $\omega(q)$ contains regular orbits and by the definition of a sectional partition, every regular orbit of
  $\omega(q)$ intersects  $int(\mathcal{R})$. Given $x\in
\omega(q)\cap int(\mathcal{R})$, we have  $x\in
int(S_j)$ for some $S_j \in \mathcal{R}'$. As $S_j$   is transverse to the flow and
$\mathcal{O}^+(q)$  accumulates  $x$  there is a sequence of points in  $\mathcal{O}^+(q)\cap int(S_j)$,  that accumulate   $x$.
Then $\mathcal{O}^+(q)\cap int(\mathcal{R})$   is an infinite set and since the cross sections in  $\mathcal{R}'$ are finite, disjoint and transverse to the flow, it follows that  $\mathcal{O}^+(p)\cap int(\mathcal{R})$     is a countable set, which  we order according to the time of return to the interior of $\mathcal{R}$.
\end{proof}

\section{Sectional partition and Sectional-hyperbolic sets}

Given $T_\Lambda M = \Df^s_\Lambda \oplus \Df^c_\Lambda$     the sectional decomposition of a sectional-hyperbolic set   $\Lambda$,
this can be extended to  $T_{U_\Lambda} M= \Df^s_{U_\Lambda} \oplus \Df^c_{U_\Lambda}$ where  $U_\Lambda$   is a neighborhood of $\Lambda$,
this extension is done continuously for  $\Df^c_U$ and integrable for $\Df^s_U$.\\

Let $\Sigma \subset U_\Lambda$ a cross section, we will denote by $\F^s_{\Sigma}$ the {\em vertical foliation} of $\Sigma$ obtained by the projection of $\F^{ss}$ over $\Sigma$ along the flow $X$, (i.e., $\F^s(x,\Sigma)$ is a leaf in $\Sigma$ obtained by the projection of the leaf $\F^{ss}(x)$ over $\Sigma$ along the flow $X$, for all $x \in \Sigma$).
  We also denote $\partial^v \Sigma$ and $\partial^h \Sigma$ the vertical and horizontal border of $\Sigma$ respectively.
  We  assume that the components of the vertical border  $\partial^v \Sigma$  are conformed by foliation leaves of $\F^s_\Sigma$, and $\partial^h \Sigma$
  is transversal to $\F^s_\Sigma$.\\

Given a sectional partition $\mathcal{R}'$, of a sectional-hyperbolic set $\Lambda$, such that $\mathcal{R} \subseteq U_\Lambda$, the foliation of $\mathcal{R}$, $\F^s_\mathcal{R}$, is determined by the foliation of the cross sections which makes contains.\\

\begin{teor}\label{teor6} Let $\omega(q)$   is a sectional-hyperbolic set of codimension  $1$, if $\omega(q)$   is not a singularity, then for all  $\alpha>0$  there is a sectional partition   $\mathcal{R}'$ of $\omega(q)$,  with diameter less than $\alpha$, such that:
\begin{enumerate}
\item[$(1)$] $int(\mathcal{R}) \cap \mathcal{O}^+(q)=\{q_1,q_2,...\}$ with $\Pi(q_{n-1})=q_n$,
\item[$(2)$] there is $\delta > 0$ and $N \in \mathbb{N}$ such that if $n \geq N$ then one of the following statements is true:\\

$(A)$ $B_\delta(q_n,\mathcal{R})\subseteq Dom(\Pi)$ and $\Pi |_{B_\delta(q_n,\mathcal{R})}$ is
$C^1$, or\\

$(B)$ $B_\delta^+(q_n,\mathcal{R})\subseteq Dom(\Pi)$ and $\Pi
|_{B_\delta^+(q_n,\mathcal{R})}$ is $C^1$,\\
where $B_\delta^+(q_n,\mathcal{R})$ denotes the connected component of $B_\delta(q_n,\mathcal{R}) \backslash s_n$
which contains $q_n$, and $s_n$ is a submanifold contained in the intersection of $\bigcup_{\sigma \in Sing(X)\cap \omega(q)}W^s(\sigma)$ with $B_\delta(q_n,\mathcal{R})$.
\end{enumerate}
\end{teor}

\begin{proof}
By  Theorem \ref{teor5} there exists a sectional partition $\mathcal{R}'$ of $\omega(q)$  of arbitrarily small diameter such that $R \subseteq U_{\omega(q)}$, and by  Lemma \ref{lema3} we have ($1$).\\

Now, by being an omega-limit sectional hyperbolic set, all singularities of $\omega(q)$ are Lorenz-like, and by  codimension $1$ we have that
 $ W^u(\sigma)$ is one-dimensional; therefore $W^s(\sigma)$   is a manifold of dimension $n-1$. On the other hand  $\mathcal{O}^+(q)\cap W^s(\sigma) = \emptyset$  for all singularities since $\omega(q)$   is not a singularity.\\

For simplicity we will denote:
\begin{align*}
A_1 & =\omega(q)\cap \mathcal{R} \cap Dom(\Pi)\\
A_2 & =\left( \omega(q) \cap int(\mathcal{R})\right)\setminus Dom(\Pi)\\
A_3 & = \left( \omega(q) \cap \partial(\mathcal{R}) \cap Cl \left[ \mathcal{O}^+(q) \cap int(\mathcal{R})\right] \right) \setminus Dom(\Pi)\\
A_4 &=\left( \omega(q) \cap \partial({\mathcal{R}})\right)\setminus \left( Dom(\Pi) \cap Cl \left[ \mathcal{O}^+(q)\cap int(\mathcal{R})\right] \right)
\end{align*}
Observe that $A_1\cup A_2 \cup A_3 \cup A_4=\omega(q)\cap \mathcal{R}$,  and besides these sets, they are pairwise disjoint. Next, to each point $x\in \omega(q) \cap \mathcal{R}$, we associate a $\delta_x>0$,   according to the set to which it belongs, as follows:

\begin{enumerate}[{\textbf{Case} 1}.]
\item If $x\in A_1$, then by the Lemma \ref{lema2} we choose $\delta_x$, such that

\begin{align} \label{ecu3}
B_{\delta_x}(x,\mathcal{R}) \subseteq Dom(\Pi) \text{ and } \Pi|_{B_{\delta_x}(x,\mathcal{R})} \text{ is } C^1
\end{align}
\end{enumerate}

For the cases  $A_2$ and $A_3$,  observe first that   $x \in \omega(q) \cap \mathcal{R} \setminus Dom(\Pi)$ implies $x \in S_j$ for some $S_j\in \mathcal{R}'$.
By  Lemma \ref{lema2}, there is $\sigma_x \in \omega(q)\cap Sing(x)$ such that $x\in W^s(\sigma_x)$. We have that $S_j$ and $W^s(\sigma_x)$ are manifolds of dimension $n-1$ and $x\in W^s(\sigma_x)\cap S_j$, in addition $W^s(\sigma_x)$ is invariant and $S_j$ is  transversal to the flow, so the connected component of  $S_j \cap W^s(\sigma_x)$ containing $x$ is a submanifold of dimension $n-2$ which we denote  $s_x$.   On the other hand, $\F^s(x)$ is a invariant manifold of dimension $n-1$, then $\F^s(x,S_j)$ is of dimension $n-2$ and we have $\F^s(x) \subseteq W^s(\sigma_x)$ and $x\in W^s(\sigma_x)$, then by Lemma   \ref{lema1}, $\F^s(x,S_j)=s_x$.\\

 Now, as $W^u(\sigma_x)$  is one-dimensional, then $W^u(\sigma_x)\backslash \{\sigma_x\}$ is divided into two connected components $W^+$ y $W^-$ and as $\mathcal{O}^+(q)\cap W^s(\sigma_x) = \emptyset$, $\mathcal{O}^+(q)$  must accumulate at least one connected component. Therefore, we have one of the following statements:
\begin{enumerate}
\item[(a)] $W^+\subseteq \omega(q)$ and $W^-\subseteq \omega(q)$;
\item[(b)] $W^+\subseteq \omega(q)$ and $W^-\nsubseteq \omega(q)$;
\item[(c)] $W^+ \nsubseteq \omega(q)$ and $W^-\subseteq \omega(q)$.
\end{enumerate}
We consider  $\beta_x$ small, such that $\mathcal{O}^+(y)$ accumulates $W^u(\sigma_x)$ for all $y \in B_{\beta_x}(x,\mathcal{R})\backslash s_x$.\\

\begin{enumerate}[{\textbf{Case} 1.}]
\addtocounter{enumi}{1}

\item If $y \in A_2$ then $y \in int(S_j)$. If the statement (a) occurs, then $W^+\subseteq \omega(q)$ and $W^-\subseteq \omega(q)$. Since $W^+$ and $W^-$  are regular orbits, by the definition of sectional partition we have  $W^+\cap int(\mathcal{R})\neq \emptyset$ and $W^-\cap int(\mathcal{R})\neq \emptyset$. Then there are $S_i , S_k\in \mathcal{R}'$ such that $W^+\cap int(S_i)\neq \emptyset$ and $W^- \cap int(S_k)\neq \emptyset$. Using the continuous dependency of the flow, there is $\delta_y<\beta_y$ small enough, so that $\mathcal{O}^+(p)  \cap int(S_i) \neq \emptyset$ or $\mathcal{O}^+(p) \cap int(S_k) \neq \emptyset$ thus $ p \in B_{\delta_y}(y,\mathcal{R})\setminus s_y$. Therefore
\begin{align}\label{ecu1}
B_{\delta_y}(y,\mathcal{R})\backslash s_y \subseteq Dom(\Pi)\,\text{  and  } \,\,\,\Pi|_{B_{\delta_y}(y,\mathcal{R}) \setminus_{s_y}} \text{  es } C^1.
\end{align}
If the statement (b) occurs, we have $W^+\subseteq \omega(q)$ and
$W^-\nsubseteq \omega(q)$. Then, there is
$S_i\in \mathcal{R}'$ such that $W^+\cap int(S_i)\neq \emptyset$. Also
$\mathcal{O}^+(q)$ does not accumulate on $W^-$.  For every $\gamma$, $s_y$, divide  $B_{\gamma}(y,\mathcal{R})$ into two connected components, we call $B_{\gamma}^+(y,\mathcal{R})$ the component that accumulates on  $W^+$ y $B_{\gamma}^-(y,\mathcal{R})$ the other; we take  $\delta_y<\beta_y$ small enough, so that $\mathcal{O}^+(q)$ does not intersect $B_{\delta_y}^-(y,\mathcal{R})$ and $\mathcal{O}^+(p)\cap int(S_i)\neq \emptyset$ for all $p\in B_{\delta_y}^+(y,\mathcal{R})$. Therefore
\begin{align}\label{ecu2}
B_{\delta_y}^+(y,\mathcal{R})\backslash s_y\subseteq Dom(\Pi)\text{ , }
\,\,\,\Pi|_{B_{\delta_y}^+(y,\mathcal{R}) \setminus s_y} \text{  is } C^1.
\end{align}
 If the statement (c) occurs, we  consider
$B_{\delta_y}^+(y,\mathcal{R})$ as the component that approaches $W^-$  and the result follows similarly to when we have (b).\\

\item If $z\in A_3$ then $z \in \partial(S_j)$. If $z \in \partial^v(S_j)$ then $s_z \subseteq \partial^v(S_j)$  and consequently $B_{\beta_z}(z,\mathcal{R})\setminus s_z$ accumulates only on one of the components $W^+$ or $W^-$. Without loss of generality we can assume that it accumulates on $W^+$. Let $S_i \in \mathcal{R}$, such that $W^+$ intersects the interior of  $S_i$, we choose $\delta_z <\beta_z$ small enough, so that $\mathcal{O}^+(p) \cap (int(S_i)) \neq \emptyset$, for all $p\in B_{\delta_z}(z,\mathcal{R}) \setminus s_z$. Then we obtain
\begin{align}
B_{\delta_z}(z,\mathcal{R})\backslash s_z \subseteq Dom(\Pi)\,\text{  and } \,\,\,\Pi|_{B_{\delta_z}(z,\mathcal{R}) \setminus_{s_z}} \text{  is } C^1.
\end{align}
Now if $z \in \partial^h(S_j)$, then $s_z$ divides $B_{\beta_z}(z,\mathcal{R})$ into two connected components, and reasoning in the same way as in the case of $A_2$, we obtain that

\begin{align}
B_{\delta_z}(z,\mathcal{R})\backslash s_z \subseteq Dom(\Pi)\,\text{  and  } \,\,\,\Pi|_{B_{\delta_z}(z,\mathcal{R}) \setminus_{s_z}} \text{  es } C^1.
\end{align}
or
\begin{align}
B_{\delta_z}^+(z,\mathcal{R})\backslash s_z\subseteq Dom(\Pi)\text{ , }
\,\,\,\Pi|_{B_{\delta_z}^+(z,\mathcal{R}) \setminus s_z} \text{  is } C^1.
\end{align}\\

\item If $w \in A_4$,  then $w \in \partial(S_j)$ for some $S_j\in \mathcal{R}'$,  we can choose $\delta_w< \frac{diam(S_j)}{2}$ such that $B_{\delta_w}(w,\mathcal{R}) \cap \left[\mathcal{O}^+(q)\cap int (\mathcal{R})\right] =\emptyset$.
\end{enumerate}
\vspace{0.5cm}

Note that $ \omega(q)\cap \mathcal{R}\backslash Dom(\Pi)$ is contained in
$$\left( \bigcup_{y_j\in A_2} B_{\frac{\delta_{y_j}}{2}}(y_j,\mathcal{R}) \right) \cup \left( \bigcup_{z_k\in A_3} B_{\frac{\delta_{z_k}}{2}}(z_k,\mathcal{R}) \right)\cup \left( \bigcup_{w_m\in A_4} B_{\frac{\delta_{w_m}}{2}}(w_m,\mathcal{R}) \right),$$
and since $\omega(q)\cap \mathcal{R} \backslash Dom(\Pi)$ is compact, then is contained in
$$\left( \bigcup_{k=1}^{l_2} B_{\frac{\delta_{y_k}}{2}}(y_k,\mathcal{R}) \right) \cup \left( \bigcup_{j=1}^{l_3} B_{\frac{\delta_{z_j}}{2}}(z_j,\mathcal{R}) \right) \cup \left( \bigcup_{m=1}^{l_4} B_{\frac{\delta_{w_m}}{2}}(w_m,\mathcal{R}) \right).$$
We define
$$ B_2  =\bigcup_{j=1}^{l_1} B_{\frac{\delta_{y_j}}{2}}(y_j,\mathcal{R}), \,\,\, B_3  = \bigcup_{k=1}^{l_3} B_{\frac{\delta_{z_k}}{2}}(z_k,\mathcal{R}), \,\,\,
B_4  = \bigcup_{m=1}^{l_4} B_{\frac{\delta_{w_m}}{2}}(w_m,\mathcal{R})$$
and
$$H  =\omega(q)\cap \mathcal{R} \backslash (B_2 \cup B_3 \cup B_4)$$
Observe that $H \subseteq \omega(q)\cap \mathcal{R} \cap Dom(\Pi)$. Thus
$$H \subseteq \bigcup_{x_i\in A_1}B_{\frac{\delta_{x_i}}{2}}(x_i,\mathcal{R}),$$
Since $H$ is compact we obtain
$$H\subseteq \bigcup_{i=1}^{l_1} B_{\frac{\beta_{x_i}}{2}}(x_i,\mathcal{R})= B_1.$$
Also, as $\omega(q)\cap \mathcal{R} \subseteq H\cup
(\omega(q)\cap \mathcal{R} \backslash Dom(\Pi))$ then $(B_1 \cup B_2\cup B_3 \cup B_4)$ is a open covering of $\omega(q)\cap \mathcal{R}$. \\

Now, as $\mathcal{O}^+(q) \cap int(\mathcal{R}) = \{q_1, q_2, ...\}$, it follows that $\{q_n\}_{n\in \mathbb{N}}$ accumulates on $\omega(q)\cap \mathcal{R}$, so there exists a large enough  $N\in\mathbb{N}$ such that for all $n>N$, we obtain $q_n \in B_1 \cup B_2 \cup B_3$, we exclude $B_4$ by the way we define
 $A_4$ and the $B_{\delta_w}(w,\mathcal{R})$.\\

Take $\delta = min \left\lbrace
\dfrac{\delta_{x_i}}{8},\dfrac{\delta_{y_j}}{8},\dfrac{\delta_{z_k}}{8} : 1\leq i \leq l_1,
1\leq j \leq l_2, 1 \leq k \leq l_3  \right\rbrace$, we have  three possibilities: $B_\delta(q_n,\mathcal{R})\subseteq B_{\delta_{x_i}}(x_i,\mathcal{R})$, $B_\delta(q_n,\mathcal{R})\subseteq B_{\delta_{y_j}}(y_j,\mathcal{R})$ or $B_\delta(q_n,\mathcal{R})\subseteq B_{\delta_{z_k}}(z_k,\mathcal{R})$.\\

If $B_\delta(q_n,\mathcal{R})\subseteq
B_{\delta_{x_i}}(x_i,\mathcal{R})$ then by \ref{ecu3}, we have to
$B_\delta(q_n,\mathcal{R})\subseteq Dom(\Pi)$ y $\Pi|_{B_\delta(q_n,\mathcal{R})}$
is $C^1$. In this case we obtain $(A)$.\\

If $B_\delta(q_n,\mathcal{R})\subseteq B_{\delta_{y_j}}(y_j,\mathcal{R})$, define
$s_n = s_{y_j}\cap B_\delta(q_n,\mathcal{R})$, $q_n \notin s_n$ because otherwise $\omega(q)$ would be a singularity. Therefore we have   $q_n \in B_\delta(q_n,\mathcal{R})\backslash s_n$. We definite $B_\delta^+(q_n,\mathcal{R})$ as the connected component of $B_\delta(q_n,\mathcal{R})\backslash s_n$ which contains $q_n$. Here we have two subcases depending on whether we have \ref{ecu1} or \ref{ecu2}.\\

If  \ref{ecu1} occurs, $B_\delta^+(q_n,\mathcal{R})\backslash s_n \subseteq
B_{\delta_{y_j}}(y_j,\mathcal{R})\backslash s_{y_j}$, therefore $B_\delta^+(q_n,\mathcal{R})\subseteq Dom(\Pi)$, and $\Pi|_{B_\delta^+(q_n,\mathcal{R})}$ is $C^1$.\\

If \ref{ecu2} is satisfied, since $q_n \in \mathcal{O}^+(q)$, then $q_n \in B_{\delta_{y_j}}^+(y_j,\mathcal{R})$, from where $B_\delta^+(q_n,\mathcal{R})\subseteq
B_{\delta_{y_j}}^+(y_j,\mathcal{R})$.
Thus
$B_\delta^+(q_n,\mathcal{R})\setminus s_n \subseteq Dom(\Pi)$ y $\Pi|_{B_\delta^+(q_n,\mathcal{R})\setminus s_n}$ is $C^1$. Then for both subcases we get $(B)$.\\

If $B_\delta(q_n,\mathcal{R})\subseteq B_{\delta_{z_k}}(z_k,\mathcal{R})$, analogously to the previous case we obtain $(B)$.\\

Then, for all $n\geq N$ we have $(A)$ or $(B)$ which proves the theorem.
\end{proof}

\section{Characterizing omega-limit sets which are closed orbits in codimension one}

\begin{defin}\label{defin4}
A point $q\in M$ satisfies the property $P_{(\Sigma)}$ if there exists an interval $I\subseteq M$ with $q \in \partial I$ and a closed set $\Sigma \subseteq M$ such that:
\begin{enumerate}
\item $Cl(\mathcal{O}^+(q)\cap \Sigma) = \emptyset$,
\item $\mathcal{O}^+(p) \cap \Sigma \neq \emptyset$ for all $p\in I$.
\end{enumerate}
\end{defin}

\begin{lema}\label{lema4}
Let $q \in M$ a point satisfying the property  $(P)_\Sigma$  for
some closed subset $\Sigma$ with $\omega(q)$ a sectional-hyperbolic set
of codimension $1$. If $\omega(q)$ is not a singularity, then there is a
sectional partition   $\mathcal{R}'$ of $\omega(q)$, $\delta>0$, $S\in \mathcal{R}'$, a sequence $\{\widehat{q}_n\}_{n\in \Dn}$ of points in $int(S) \cap \mathcal{O}^+(q)$ and a sequence of intervals $\widehat{J}_1,\widehat{J}_2,...\subseteq
S$ in the positive orbit of $I$ with $\widehat{q_i} \in
\partial(\widehat{J_i})$ and $l(\widehat{J_i})\geq \delta$ for all $i$.
\end{lema}

\begin{proof}
Without loss of generality we can assume that $q\in U_{\omega(q)}$ and the arc  $I$ that refers to property  $(P)_\Sigma$, is tangent to $\Df^c$ and transverse to the flow. Since  $Cl(\mathcal{O}^+(q))$ and $\Sigma$  are disjoint,
there exists a compact neighborhood $W\subseteq U_{\omega(q)}$ of $\omega(q)$
such that $W\cap \Sigma = \emptyset$ and $\mathcal{O}^+(q)\subseteq W$, then by the theorem  \ref{teor6}, we have a sectional partition $\mathcal{R}'=\{S_1, S_2, ... , S_k\}$ of $\omega(q)$ contained in $int(W)$ (since we can take the diameter of  $\mathcal{R}$ arbitrarily small), and $N\in \Dn$ such that $\mathcal{O}^+(q)\cap int(\mathcal{R})=\{q_1,q_2,...\}$, and for all $n \geq N$, $q_n$ satisfies  $(A)$ or $(B)$. We will assume $N=1$.\\

For all $n$ there is  $S_{j_n}\in \mathcal{R}'$, such that $q_n \in int(S_{j_n})$. As $q \in \partial(I)$, for the continuous dependence we have to $q_n$  must be a border point of the positive orbit of   $I$, and since $S_{j_n}$ is transverse to the flow as well as $I$, we can guarantee that exists $I_1$ in the positive orbit of   $I$ such that:

$$I_1\subseteq S_{j_1}\cap Dom(\Pi) \,\,\, \text{ and }\,\,\,
q_1\in \partial(I_1).$$

If necessary we can reduce it to  $I$ so that $I_1\subseteq
int(B_\delta(q_1,\mathcal{R}))$ or $I_1\subseteq int(B_\delta^+(q_1,\mathcal{R}))$   depending on
whether  occurs  $(A)$ or $(B)$ for  $q_1$; we define $I_i = \Pi(I_{i-1})
= \Pi^i(I_1)$ for $i>1$ and while  $I_{i-1}\subseteq
B_\delta(q_{i-1},\mathcal{R})$ or $I_{i-1}\subseteq B_\delta^+(q_{i-1},\mathcal{R})$ again depending on whether occurs $(A)$ or $(B)$ for $q_{i-1}$. Since $W \cap \Sigma = \emptyset$ and the positive orbit of $I$ intersects $\Sigma$, there exists a first index  $i_1$ such that:
$$I_{i_1} \nsubseteq B_\delta(q_{i_1},\mathcal{R})\,\,\, \text{ or } \,\,\, I_{i_1}
\nsubseteq B_\delta^+(q_{i_1},\mathcal{R}).$$
We define $J_{i_1} \subseteq I_{i_1}$ as the connected component of
$I_{i_1} \cap B_\delta(q_{i_1},\mathcal{R})$   (or $I_{i_1} \cap B_\delta^+(q_{i_1},\mathcal{R})$) which bounded with $q_{i_1}$, and some point
in  $\partial (B_\delta(q_{i_1} \mathcal{R}))$ (or in $\partial(B_\delta^+(q_{i_1},\mathcal{R}))$. Remember that
$s_{i_1}\subseteq \bigcup_{\sigma \in Sing(X)\cap \omega(q)}W^s(\sigma)$ and
$\mathcal{O}^+(I_{i_1})\cap \Sigma \neq \emptyset$ we have  $I_{i_1}\cap s_{i_1} = \emptyset$ and we can conclude that
$l(J_{i_1}) \geq \delta$.\\

Now, if necessary, we reduce  $I_{i_1}$ so that
$\Pi(I_{i_1}) \subseteq B_\delta(q_{i_1+1},\mathcal{R})$  (or $ \Pi(I_{i_1})
\subseteq B_\delta^+(q_{i_1+1},\mathcal{R})$) and repeat the argument used in
$I_1$ for $I_{i_1}$  and in this way find an index $i_2$
and construct the interval  $J_{i_2} \subseteq I_{i_2}$ satisfies same conditions of $J_{i_1}$, then we would reduce if necessary to $I_{i_2}$   to repeat the process and so on, then we get a sequence  $\{J_{i_m}\}_{m\in \mathbb{N}}$ such that $J_{i_m}\subseteq S_{j_{i_m}}$, $q_{i_m}\in
\partial (J_{i_m})$ y $l(J_{i_m})\geq
\delta$ for all $m$.\\

Since $\mathcal{R}$  is a finite collection we have that this sequence has a subsequence $\{J_{i_{m_s}}\}_{s\in \mathbb{N}}$ such that
$J_{i_{m_s}}\subseteq S_r$ for some $S_r \in \mathcal{R}'$, considering $S=S_r$, $\widehat{J}_s=J_{i_{m_s}} $ and
$\widehat{q}_s=q_{i_{m_s}}$ the result is obtained.
\end{proof}

\begin{teor}\label{teor7}
Let $q \in M$ be a point satisfying the property $(P)_\Sigma$ for
some subset $\Sigma$ closed and $\omega(q)$ is a  sectional hyperbolic codimension  $1$ set. If $\omega(q)$ is not a singularity, then  $\omega(q)$ is a periodic orbit.
\end{teor}

\begin{proof}
Let $W$ and $\mathcal{R}'$ as in the proof of the Lemma    \ref{lema4}, since $\omega(q)$   is not a singularity, then  there is $S \in \mathcal{R}'$,
$\widehat{q_i}$, $\widehat{J_i}$ and $\delta>0$ such that
$\widehat{q_i}\in int(S) \cap \mathcal{O}^+(q)$, $\widehat{J_i} \subseteq
\mathcal{O}^+(I) \cap S$, $\widehat{q_i}\in \partial \widehat{J_i}$, and $l(\widehat{J_i})\geq \delta$ for all
$i\in \mathbb{N}$, also, suppose there is $x \in \omega(q) \cap S$, such that $\widehat{q_i}$ accumulates on $x$.\\

If $x \in \partial^v(S)$, then $\widehat{q_i} \notin \mathcal{F}^s(x,S) \text{ for all } i$, since $ \mathcal{F}^s(x,S) \in \partial(S)$. Since $\widehat{j_i}$ is tangent a $\Df^c_U$ and transverse to $X$
then the angle between the arc $\widehat{j_i}$ and $\mathcal{F}^s_\Sigma$
is bounded away from zero for all  $i$, also as
$l(\widehat{J_i}) \geq \delta$ and $\widehat{q_i}\to x$ there will eventually
 $z$ such that:
$$z \in \widehat{J_r}\cap \mathcal{F}^s(\widehat{q_j},S)$$
for some pair $r, k  \in \mathbb{N}$. Since $z\in \widehat{J_r} \subseteq \mathcal{O}^+(I)$, then $\mathcal{O}^+(z)\cap \Sigma \neq \emptyset$, on the other hand $z \in \mathcal{F}^s(\widehat{q_j},S)$ so
$\mathcal{O}^+(z)$ is asymptotic at $\mathcal{O}^+(q)$ and as $W$ is compact, then
 $\mathcal{O}^+(z) \subseteq W$ so $\mathcal{O}^+(z) \cap \Sigma = \emptyset$; which is a contradiction, therefore $x\notin \partial^v(S)$.\\

Now, if $x \in \partial^h(S)$ or $x\in int(S)$, we have that $\{\widehat{q_1}, \widehat{q_2},...\} \setminus \F^s(x,S)$ has a finite number of elements, otherwise we would find again that there is $z \in \widehat{J_r}\cap \mathcal{F}^s(\widehat{q_j},S)$ for some pair $r, k  \in \mathbb{N}$ and we have a contradiction. Thus $\{\widehat{q_1},\widehat{q_2}, ...\} \cap \mathcal{F}^s(x,S)$ is a infinite set and
 we can organize  as a succession $\{q_n\}_{n\in \mathbb{N}}$, such that $q_i$ belongs to positive orbit of $q_{i-1}$,  and so, the hypotheses of the Lemma 11\footnote{Although this is presented in dimension three, it is valid in arbitrary dimension} in \cite{s} are satisfied, then there is $p \in Per(X)\cap \omega(q)$ such that $q_n \in \F^s(p)$ but as $q_n \in \mathcal{O}^+(q)$
therefore we can conclude that $\omega(q) = \gamma=\mathcal{O}(p)$.
\end{proof}

\begin{teor}\label{teor8}
Let $q \in M$, such that $\omega(q)$  is  a sectional-hyperbolic set of codimension   $1$. If $\omega(q)$ is a closed orbit, then  $q$ satisfies
the property  $(P)_\Sigma$   for some closed subset  $\Sigma$.
\end{teor}

\begin{proof}
See Theorem 4 in \cite{s}
\end{proof}

As a direct consequence of the Theorems \ref{teor7} and \ref{teor8}, we obtain:

\begin{teor}\label{teor9}
Let $\omega(q)$ a sectional-hyperbolic set of codimension  $1$ of a vector field $X$ on $M$, then $q$ satisfies the property  $P_{(\Sigma)}$  if and only if $\omega(q)$ is a closed orbit.
\end{teor}

\section{Proof the main theorem}

For the proof, we will first analyze two particular cases and then the general case.

\begin{teor}\label{teor10}
Let $\Lambda$   be a sectional-hyperbolic set of codimension $1$ from a field $X$ on $M$, such that $W^{u}(H) \subseteq \Lambda$ for every hyperbolic subset $H$ of $\Lambda$. If $p,\sigma \in \Lambda$ satisfy $p \prec \sigma$ where $p$ a periodic point and  $\sigma$ a singularity, then there exists $x\in \Lambda$ such that $\alpha(x)=\alpha(p)$ and $\omega(x)$ is a singularity.
\end{teor}

\begin{proof}
Since $p \prec q$, then there are successions $(z_n)_{n \in \Dn}$ with $z_n \to p$, and $(t_n)_{n\in \Dn}$ with $t_n > 0$ such that $X_{t_n}(z_n) \to \sigma$; we can take $t_n \to \infty$    and without loss of generality we can assume that  $z_n \in U_\Lambda$ for all $n \in \Dn$.
 We denote by $O=\mathcal{O}(p)$  the periodic orbit containing  $p$, as $O \subseteq \Lambda$ is hyperbolic we have  $W^{uu}(p)$ is well defined and by hypothesis contained in  $\Lambda$, in addition $\F^{ss}(p)=W^{ss}(p)$. Now for the continuity of  $\F^{ss}(p)$ and given that $z_n \to p$, for $n$ sufficiently large $\F^{ss}(z_n)$ intersects $W^{uu}(p)$ at a point $z'_n$. Since $z_n$ and $z'_n$ have the same strong stable manifold, $X_{t_n}(z_n) \to \sigma$ and $t_n \to \infty$, we have $X_{t_n}(z'_n) \to \sigma$. But $z'_n \in W^u(O)$ which is invariant, then  $\sigma \in Cl(W^u(O))$, therefore, $\sigma$ is Lorenz-like.\\

We choose two cross-sections $\Sigma_1 \subseteq U_\Lambda$ and $\Sigma_2 \subseteq U_\Lambda$, associated with $\sigma$ such that the intersection of $\Sigma_1$ with   one of the connected components of $W^{ss}(\sigma)\setminus \F^{ss}(\sigma)$ is a point $y_1 \in int(\Sigma_1)$ and the intersection of $\Sigma_2$ with the other component is a point   $y_2 \in int(\Sigma_2)$. We take this sections of small size, so that $O \cap (\Sigma_1 \cup \Sigma_2) = \emptyset$. Since $\Lambda \cap \F^{ss}(\sigma)=\{\sigma\}$, then we can establish  $(\partial^h\Sigma_1 \cup \partial^h\Sigma_2)\cap \Lambda = \emptyset$. Let  $\F^s_{\Sigma_1}$ y $\F^s_{\Sigma_2}$ the vertical foliations of $\Sigma_1$ and $\Sigma_2$.\\

On the other hand as $W^u(O)$ accumulates on $\sigma$ then accumulates on $\F^s(y_1,\Sigma_1)$, $\F^s(y_2,\Sigma_2)$ or both. Assume without loss of generality that accumulates on  $\F^s(y_1,\Sigma_1)$. We can select a point $c \in W^{uu}(p) \cap int(\Sigma_1)$, as the first point where  $W^{uu}(p)$ that intersects $int(\Sigma_1)$.  Taking the negative orbit of  $c$  and using the fact that  $dim(W^{uu}(p))=1$, we obtain a fundamental domain $D^u = [a,b]$ of $W^{uu}(p)$ such that $D^u \cap \Sigma = \emptyset$. Furthermore, $b$  is in the positive orbit of  $a$  and in the negative orbit of  $c$.\\

We define the function $\Pi_D: Dom(\Pi_D) \subseteq D^u \to int(\Sigma_1)$
with
$$Dom(\Pi_D)=\{x \in D^u: X_t(x)\in int(\Sigma_1)
\text{ for some } t > 0\}$$ given by $\Pi_D(x) = X_{t(x)}(x)$, where $t(x)$  is the return time, that is, the first  $t>0$ for which
$X_t(x)\in int(\Sigma_1)$.\\

By construction you have $a, b \in Dom(\Pi_D)$, and since $b$   is in the positive orbit of  $a$  it follows that $\Pi_D(a) = \Pi_D(b) = c \in int(\Sigma_1)$. Define\\
\begin{align*} q^*=
 Sup \{s \in [a,b] : [a,s] \subseteq Dom(\Pi_D),  \Pi_D([a,s]) & \subseteq int(\Sigma) \\ & \text{ and } \Pi_D|_{[a,s]} \text{ is } C^1 \}
 \end{align*}
\begin{align*} q^{**}=
Inf \{s \in [a,b] : [s,b] \subseteq Dom(\Pi_D), \Pi_D([s,b]) & \subseteq int(\Sigma) \\ & \text{ and }  \Pi_D|_{[s,b]} \text{ is } C^1 \}
\end{align*}

Since $\Pi_D(a), \Pi_D(b) \in int(\Sigma)$,  by the continuous dependence of the flow, $q^*$ y $q^{**}$  are well defined and $a < q^*$ y $q^{**} < b$; now if $q^*=b$, $q^{**}=a$ o $q^*=q^{**}$, we would  have  $\Pi_D([a,b])$  is a curve closed  $l$ in $int(\Sigma_1)$  (without a point in the third case) and therefore tangent to $\F^s_{\Sigma_1}$   in at least one point, but as  $D^u \subseteq W^{uu}(p)$, then the vectors tangent to $l$  belong to the central subspace $\Df^c$, meaning that $l$  is transversal to the strong stable foliation in  $\Lambda$
 and therefore to the foliation  $\F^s_{\Sigma_1}$,  which contradicts that it is closed, then  $a<q^*<q^{**}<b$,  also again by the continuous dependency we have that  $q^*,q^{**} \notin Dom(\Pi_D)$, otherwise $q^*$  would not be a supremum or  $q^{**}$ an infimum.\\

Let $c^*,c^{**}\in int(\Sigma_1)$, the open extreme of the semi-open arcs $\Pi_D([a,q^*))$ and $\Pi_D((q^{**},b])$ respectively and $l =\Pi_D([a,q^*)\cup (q^{**},b])$, as $\Pi_D(a)=\Pi_D(b)=c$, we have  $l$, it is an open connected arc with ends  $c^*$ y $c^{**}$, in addition $([a,q^*)\cup (q^{**},b])\subseteq W^{uu}(p)$, then $l$ is transverse to $\F^s_{\Sigma_1}$. On the other hand, since  $\Lambda$ has codimension  $1$, $\F^s(y_1,\Sigma_1)$ divide a $\Sigma_1$ into two connected components. Two cases are presented:

\begin{enumerate}
\item[{\bf Case 1:}] {If $c^*$ and $c^{**}$ are in different sides of $\F^s(y_1,\Sigma_1)$}, then, the arc $l$ intersects the leaf $\F^s(y_1,\Sigma_1)$, that is, there exists $x \in W^{s}(\sigma)\cap W^{uu}(p)$, then $\alpha(x)=\alpha(p)$ and $\omega(x) = \{\sigma\}$.\\

\item[{\bf Case 2:}] {If  $c^*$ and $c^{**}$ are on the same side of $\F^s(y_1,\Sigma_1)$, or $c^* \in \F^s(y_1,\Sigma_1)$ or $c^{**} \in \F^s(y_1,\Sigma_1)$}. Assume without loss of generality that $c^*$ is closer to $\F^s(y_1,\Sigma_1)$, than $c^{**}$.  Consider the cross section $\Sigma_0 \subseteq \Sigma_1$, given by the leaves $\F^s(y_1,\Sigma_1)$, $\F^s(\Pi_D(r),\Sigma_1)$, and all the leaves $\F^s_{\Sigma_1}$ between them, at where $r \in (q^{**},b)$. We have that $\mathcal{O}^+(q^*)$
does not intersect the interior of $\Sigma_1$ and $\partial^h\Sigma_0 \cap \Lambda = \emptyset$, then $q^*$ satisfies the property $P_{(\Sigma_0)}$ by taking $I= (a,q^*)$.\\

Since $q^* \in \Lambda$ sectional-hyperbolic,  $\omega(q^*)$ is also sectional-hyperbolic, then by the theorem \ref{teor8}, $\omega(q^*)$
is a periodic orbit or a singularity. If  $\omega(q^*) = \mathcal{O}(\widehat{p})$ with $\widehat{p} \in \Lambda \cap Per(X)$, then follows from Inclination Lemma (see lemma 2.15 in \cite{ap}),  that the positive orbit    $I$ accumulates on $W^u(\mathcal{O}(\widehat{p}))$. In particular, the positive orbit of $I$ contains an open arc $I^*$ arbitrarily close to $\widehat{D^u}=[\widehat{a},\widehat{b}]$, where $\widehat{D^u}$ is a fundamental domain of  $W^{uu}(\widehat{p})$.\\

We define the function $\Pi_{\widehat{D}}: Dom(\Pi_{\widehat{D}}) \subseteq \widehat{D^u} \to int(\Sigma_1)$
with
$$Dom(\Pi_{\widehat{D}})=\{x \in \widehat{D^u}: X_t(x)\in int(\Sigma_1)
\text{ for some } t > 0\}$$ given by $\Pi_{\widehat{D}}(x) = X_{t(x)}(x)$, where $t(x)$  is the return time, that is, the first $t>0$ for which
$X_t(x)\in int(\Sigma_1)$.\\

By projecting $I^*$ onto $\widehat{D^u}$  along the strong stable manifolds of the points in  $I^*$,  we can conclude that  $\widehat{D^u} \subseteq Dom(\Pi_{\widehat{D}})$, then $\Pi_{\widehat{D}}(\widehat{D^u})$   is a closed curve $\widehat{l} \subseteq int(\Sigma_1)$, but since $\widehat{D^u}$
is a fundamental domain of  $W^{uu}(\widehat{p}) \subseteq \Lambda$, then $\widehat{l}$  is transversal to $\F^s_{\Sigma_1}$, which contradicts that  $\widehat{l}$    is closed. From the above contradiction we conclude that $\omega(q^*)$  is a singularity, therefore  $q^* \in W^s(\sigma^*) \cap W^{uu}(p)$   for some singularity $\sigma^* \in \Lambda$.  Then taking $x = q^*$, we get the result.
\end{enumerate}

\end{proof}

\begin{teor}\label{teor11}
Let $\Lambda$ be a sectional-hyperbolic set of codimension  $1$ from a field $X$ over $M$, such that $W^{u}(H) \subseteq \Lambda$ for every hyperbolic subset $H$ of $\Lambda$. If $p,\sigma \in \Lambda$ satisfies $p \prec \sigma$, $\alpha(p)$ does not contain singularities and $\sigma$   is a singularity, then there exists  $x\in M$ such that $\alpha(x)=\alpha(p)$ and $\omega(x)$ is a singularity.
\end{teor}

\begin{proof}
We have that $\alpha(p)$ hyperbolic since it has no singularities. We fix $y \in \alpha(p)$, then there is a sequence  $(t_n)_{n\in \Dn}$ with $t_n \to \infty$ such that $X_{-t_n} (p) \to y$. We extend the hyperbolic decomposition of  $\alpha(p)$ to a neighborhood $U_{\alpha(p)}$,  as the negative orbit of $p$ becomes close to  $\alpha(p)$,  we can assume that $X_{-t_n} (p) \in U_{\alpha(p)}$, we can use graphics transformation techniques  \cite{hps,hps1},  to find a $\epsilon > 0$ and a succession of open intervals $(I_n)_{n\in \Dn}$ where $I_n = (X_{-t_n }(p) - \epsilon, X_{-t_n} (p) + \epsilon) \subseteq W^{uu}(X_{-t_n} (p))$, converging to the open interval $I = (y - \epsilon, y + \epsilon) \subseteq W^{uu}(y)$. \\

On the other hand, by applying the Shadowing Lemma \cite{kh} to the negative orbit of  $p$, we can establish a succession of periodic hyperbolic points  $\{p_n\}_{n\in \Dn}$ so that  $p_n \to y$ and $p_n \in U_{\alpha(p)}$. For $n$   large enough, by the continuity  $W^{ss}_{U_{\alpha(p)}}$,  we have $W^{ss}(p_n)$ intersect $W^{uu}(y)$ at a point $q_n$, given that $p_n$ and $q_n$ have the same strong stable manifold $\omega(q_n) = \omega(p_n)= \mathcal{O}(p_n)$ and as $q_n \in W^{uu}(y) \subseteq \Lambda$, then $p_n \in \Lambda$. \\

In addition, strong unstable manifolds $W^{uu}(p_n)$   have uniformly large size and approach to $I$ when $n \to \infty$.  Then both $W^{uu}(p_n)$ and $I_n$ approximate the interval  $I$ when $n \to \infty$; this allows us to fix $n_0,n_1 \in \Dn$ such that $p_{n_1} \in \Lambda$   and satisfies the property:\\

\textbf{Property $(Q)$}  The strong stable manifold of each point close to   $X_{-t_{n_0}}(p)$  intersects  $W^{uu}(p_{n_1})$,  and conversely, the strong stable manifold of any point close to $p_{n_1}$ intersects $I_{n_0}$\\

Now, since $p \prec \sigma$, so we also have to $X_{-t_{n_0}} (p) \prec \sigma$,  from where, there are successions $(z_m)_{m\in \Dn}$ con $z_m \to X_{-t_{n_0}} (p)$ and $(t_m)_{m \in \Dn}$ with $t_m > 0$ such that $X_{t_m} (z_m) \to \sigma$. Then the property  $(Q)$ implies that there is another sequence  $(z_m')_{m \in \Dn} \subseteq  W^{uu}(p_{n_1} ) \cap W^{ss}(z_m)$; then $X_{t_m}(z_m') \to \sigma$, therefore $p_{n_1} \prec \sigma$.
Applying the theorem \ref{teor10}, we have that exists $x^* \in \Lambda$ such that $\alpha(x^*) = \alpha(p_{n_1})$ y $\omega(x^*)$ is a singularity  $\sigma^*$.\\

Taking the negative orbit of $x^*$, we can assume that $x^*$  is close enough to $p_{n_1}$.  Then the property  $(Q)$ implies that $W^{ss}(x^*)$ intersects $I_{n_0}$  at some point $x$. Then as $I_{n_0} \subseteq W^{uu}(X_{-t_{n_0}}(p))$ and $\alpha(X_{-t_{n_0}} (p)) = \alpha(p)$,  then you have $\alpha(x) = \alpha(p)$, we have $x \in W^{ss}(x^*)$ then$\omega(x) = \omega(x^*) = \sigma^*$, which proves the result.
\end{proof}

\begin{proof}[Proof main theorem]
The result is immediate if $q \in \mathcal{O}^+(p)$,  then assume that $q \in \mathcal{O}^+(p)$. If $\omega(p)$ or $\omega(q)$ contain a singularity  $\sigma$, then $p \prec \sigma$, similarly if $\alpha(q)$ contains a singularity $\sigma$, then the continuity of the flow $X_t$, and the fact that $q \notin \mathcal{O}^+(p)$  implies that $p \prec \sigma$; then the result follows from the previous theorem.\\
We will assume that the set  $\alpha(p) \cup \omega(p) \cup \alpha(q) \cup \omega(q)$ has no singularities, then there is $\delta_1 > 0$ such that $p$ and $q$ are in $H_1$ defined by:

$$H_1 = \bigcap_{t \in \Dr} X_t(\Lambda \setminus B_{\delta_1} (Sing(X))).$$

Since $H_1$  has no singularities, then it is a hyperbolic set, for which $W^{uu}(p)$ is well defined,  and by hypothesis contained in  $\Lambda$, reasoning analogously to the beginning of the proof of theorem \ref{teor10}, there exists a sequence $\{z'_n\}_{n \in \Dn}$, such that $z_n' \in W^{uu}(p)$ and $X_{t_n}(z'_n) \to q$ with $z'_n \to p$ and $t_n \to  \infty$.\\

Suppose for a moment that for all $k \in \Dn$ there  is $\sigma_k \in Sing(X)$ such that
$$\sigma_k \in Cl\left( \bigcup_{n=k}^\infty \mathcal{O}^+(z'_n)\right)$$

Since the number of singularities in  $\Lambda$ is finite, we can assume that  $\sigma = \sigma_k$ does not depend on $k$. As $z_n' \to p$,  then it is concluded that  $p \prec \sigma$; then the result follows from the theorem  \ref{teor11}. Then we can assume that there is  $k_0 \in \Dn$ y $0 < \delta_2 < \delta_1$ such that
$$ \left( \bigcup_{n=k_0}^\infty \mathcal{O}^+(z'_n) \right) \cap B_{\delta_2}(Sing(X))=\emptyset$$
Observe that $\mathcal{O}(z_n') \subseteq \Lambda$,
for which,
$$\mathcal{O}^+(z_{n_0}) \subseteq \Lambda \setminus B_{\delta_2}(Sing(X))$$
for all $n \geq k_0$.
On the other hand, $z'_n \in  W^{uu}(p)$, then $\alpha(z'_n) = \alpha(p)$ that has no singularities. Then, like $z'_{n_0} \to p$,   we conclude that there exists  $\delta_3 < \delta_2$ such that $\mathcal{O}^-(z'_n) \subseteq U \setminus B_{\delta_3}(Sing(X))$, for all $n \geq k_0$. Consequently $(z_n')_{n \geq k_0} \subseteq H$ where:
$$H=\bigcap_{t \in \Dr}X_t(U \setminus B_{\delta_3}(Sing(X)))$$
 which has no singularities and therefore is hyperbolic. Then as $X_{t_n} (z_n') \to q$ y $H$   is hyperbolic, by the theorem  \ref{teor2}, there is $x \in M$ such that $\alpha(x) = \alpha(p)$ and $\omega(x) = \omega(q)$.
\end{proof}


\begin{thebibliography}{HH}

\bibitem{ap}
Araujo, V., Pacifico, M. J. {\em Three-dimensional flows}, Springer-Verlag, Berlin, (2010).

\bibitem{am}
Arbieto, A., Morales, C.A.,
A dichotomy for higher-dimensional flows,
{\em Proc. Amer. Math. Soc.} 141 (2013), no. 8, 2817--2827.

\bibitem{aml}
Arbieto, A., Morales, C. A., Lopez, A. M.,
Homoclinic classes for sectional-hyperbolic sets,
{\em Kyoto Journal of Mathematics.} Vol. 56 (2016) no. 3, 531--538


\bibitem{bm}
Bautista, S., Morales, C.A.,
{\em Lectures on sectional-Anosov flows}. Preprint IMPA Serie D 84 (2011).

\bibitem{bm1}
Bautista, S.,  Morales, C.A.,
A sectional-Anosov connecting lemma,
{\em Ergodic Theory Dynam. Systems} 30 (2010), no. 2, 339--359.

\bibitem{bm2}
Bautista, S., Morales, C.A.,
Characterizing omega-limit sets which are closed orbits,
{\em J. Differential Equations} 245 (2008), no. 3, 637--652.


\bibitem{hps1}
Hirsch, M., W., Pugh, C., C., Shub, M.,. Neighborhoods of hyperbolic sets. {\em Invent. Math}. Vol. 9 (1969/1970), 121-134.


\bibitem{hps}
Hirsch, M., W., Pugh, C., C., Shub, M.,
Invariant manifolds,
{\em Lecture Notes in Mathematics}, Vol. 583. Springer-Verlag, Berlin-New York, (1977).



\bibitem{kh}
Katok, A., Hasselblatt, B.,
{\em Introduction to the modern theory of dynamical systems}.
With a supplementary chapter by Katok and Leonardo Mendoza.
Encyclopedia of Mathematics and its Applications, 54. Cambridge
University Press, Cambridge, (1995).

\bibitem{mp2}
Morales, C.A., Pacifico, M.J.,
A dichotomy for three-dimensional vector fields,
{\em Ergodic Theory Dynam. Systems} 23 (2003), no. 5, 1575--1600.


\bibitem{s}
Sanchez, Y., Caracterización de conjuntos omega-límites singulares hiperbólicos que son órbitas cerradas. Master's degree work. {\em Universidad Nacional de Colombia.} Bogotá (2012). http://www.bdigital.unal.edu.co/6951/1/830370.2012.pdf

\end{thebibliography}
\end{document}